\documentclass[twoside,11pt]{article}
\usepackage{latexsym}
\usepackage{epsfig}
\usepackage{amsfonts}
\usepackage{amsthm,amsmath,amsfonts,amssymb}
\numberwithin{equation}{section}

\begin{document}
\newtheorem{assumption}{Assumption}
\newtheorem{proposition}{Proposition}
\newtheorem{definition}{Definition}
\newtheorem{lemma}{Lemma}
\newtheorem{theorem}{Theorem}
\newtheorem{observation}{Observation}
\newtheorem{remark}{Remark}
\newtheorem{corollary}{Corollary}

\title{Ellipticity of quantum mechanical Hamiltonians in the edge algebra}

\author{Heinz-J{\"u}rgen Flad$^\dag$ and Gohar Harutyunyan$^\ast$ \\
\ \\
$^\dag${\small Institut f\"ur Mathematik, Technische Universit\"at Berlin,
Stra{\ss}e des 17. Juni 136, D-10623 Berlin}\\
$^\ast${\small Institut f\"ur Mathematik, Carl von Ossietzky Universit\"at
Oldenburg, D-26111 Oldenburg}\\
}
\maketitle

\begin{abstract}
We have studied the ellipticity of quantum mechanical Hamiltonians, in particular
of the helium atom, in order to prove existence of a parametrix and corresponding
Green operator. The parametrix is considered in local neighbourhoods of coalescence
points of two particles. We introduce appropriate hyperspherical coordinates where
the singularities of the Coulomb potential are considered as embedded edge/corner-type
singularities. This shows that the Hamiltonian can be written as an edge/corner 
degenerate differential operator in a pseudo-differential operator algebra.
In the edge degenerate case, we prove the ellipticity of the Hamiltonian.
\end{abstract}

\section{Introduction}
A principal task of quantum chemistry is the development of many-particle models in
electronic structure theory which enable accurate predictions of molecular properties.
Despite tremendous achievements in this field, cf.~\cite{HJO}, considerable further efforts are necessary
in order to meet the standards of present-day experiments. In respect to this it is important
to know the asymptotic behaviour of these models near coalescence points of particles.
It is the purpose of our work to develop tools which help to deepen our understanding
of the asymptotic behaviour and can be eventually used to improve such models and corresponding
numerical solution schemes. The Coulomb singularities at coalescence points of particles
are treated as embedded conical, edge and corner singularities in the configuration space of electrons.
Our approach consists of a recursive construction of the parametrix and corresponding Green operator
of a nonrelativistic Hamiltonian. A peculiar feature of the singular pseudo-differential calculus is that
the Green operator encodes all the asymptotic information of the eigenfunctions of the Hamiltonian.
In our previous work \cite{FHSchSch}, we have explicitly calculated the asymptotic Green operator
of the hydrogen atom. A further important step is the application of our techniques to the helium
atom which represents the simplest nontrivial example of a many-electron system.
The present work sets the ground for an application of the singular pseudo-differential calculus
to the helium atom and demonstrates ellipticity of the Hamiltonian as an edge degenerate operator
in a local neigbourhood of the coalescence points of two particles. This proves existence of a parametrix
and corresponding Green operator in this case.

We study the Hamiltonian for a three-particle system corresponding to the helium
atom in the
Born-Oppenheimer approximation which means that the nucleus is kept fixed at the
origin. For the helium atom, the Hamiltonian is given by 
\begin{equation}
 H := -\frac{1}{2} \bigl( \Delta_1 + \Delta_2 \bigr)
 - \frac{2}{|{\bf x}_1|} - \frac{2}{|{\bf x}_2|} + \frac{1}{|{\bf x}_1 - {\bf x}_2|} ,
 \label{HeH}
\end{equation}
where ${\bf x}_1 := (x_1,x_2,x_3), {\bf x}_2 := (x_4,x_5,x_6)$ are vectors in
$\mathbb{R}^3$ denoting
the spatial coordinates of the two electrons. Furthermore we have used the
physicists notation
for the Laplacian in $\mathbb{R}^6$, i.e.,
\begin{equation}
 \Delta_1 + \Delta_2 := \frac{\partial^2}{\partial x^2_1} +
\frac{\partial^2}{\partial x^2_2} +
 \frac{\partial^2}{\partial x^2_3} + \frac{\partial^2}{\partial x^2_4} +
 \frac{\partial^2}{\partial x^2_5} + \frac{\partial^2}{\partial x^2_6} .
\label{Laplace}
\end{equation}
The Hamiltonian (\ref{HeH}) can be either considered as a bounded operator from
$H^1(\mathbb{R}^6)$ into $H^{-1}(\mathbb{R}^6)$,
or as an unbounded essentially self-adjoint operator on $L^2(\mathbb{R}^6)$ with
domain
$H^2(\mathbb{R}^6)$, see e.g.~Kato's monograph \cite{Kato}. We want to discuss this
operator
from the point of view of edge and corner degenerate operators \cite{Schulze02}.
Our main problem is to incorporate the singular Coulomb potentials in an
appropriate manner,
which lead to edge and corner like embedded singularities. These  are denoted as {\em electron-nuclear} ($e-n$), {\em
electron-electron} ($e-e$),
and {\em electron-electron-nuclear} ($e-e-n$) singularities.
We want to demonstrate
in the following that    $e-n$, $e-e$ and $e-e-n$ singularities correspond to edges and a corner, respectively, in our singular calculus.

\section{The helium atom from the point of view of singular analysis}
\subsection{Hierarchy of edge and corner singularities}
Our discussion starts with the configuration space $\mathbb{R}^6$ of two electrons,
which we
consider in the following as a stratified manifold with embedded edge and corner
singularities. Let us introduce
polar coordinates in $\mathbb{R}^6$ with radial variable
\begin{equation}
 t:= \sqrt{x_1^2 + x_2^2 + x_3^2 + x_4^2 + x_5^2 + x_6^2} .
\label{t-def}
\end{equation}
This means we can formally consider $\mathbb{R}^6$ as a conical manifold with
embedded conical singularity at the origin, i.e.,
\[
 \mathbb{R}^6 \equiv (S^5)^\Delta:=(\overline{\mathbb{R}}_+ \times S^5) / (\{0\} \times S^5) ,
\]
with base $S^5$.  Removal of this singular point defines an open
stretched cone
\[
 (S^5)^{\wedge} := \mathbb{R}_+ \times S^5 .
\]

To the singularities of the Coulomb potential
\[
|\boldsymbol{x}_1|=0,\quad |\boldsymbol{x}_2|=0, \quad |\boldsymbol{x}_1-\boldsymbol{x}_2|=0
\]
correspond  closed embedded disjoint submanifolds $Y_1, Y_2,$$ Y_3$ of $S^5$, which are homeomorphic to $S^2$ and can be considered as  edge singularities.

For any of the submanifolds
$Y_i$, $i=1,2,3$,
there exists a local neighbourhood $U_i$ on $S^5$ which is homeomorphic to a wedge
\[
 W_i = X^{\Delta}_i \times Y_i \ \ \mbox{with} \ X^{\Delta}_i :=
 (\overline{\mathbb{R}}_+ \times X_i) / (\{0\} \times X_i) ,
\]
where the base $X_i$ of the wedge is again homeomorphic to $S^2$. The associated
open stretched wedges
are
\[
 \mathbb{W}_i = X^{\wedge}_i \times Y_i \ \ \mbox{with} \ X^{\wedge}_i :=
\mathbb{R}_+ \times X_i .
\]

The conical manifold $(S^5)^\Delta$ considered with (local) wedges $W_i, i=1,2,3,$ on its basis $S
^5$ becomes a manifold with (higher order) corner singularity at the origin. 

\subsection{Local coordinates near singularities}
In order to interpret the singularities of the Coulomb potential as embedded
singularities in $\mathbb{R}^6$,
 we have to specify
local coordinates in appropriate neigbourhoods of  edges on $S^5$ which provide
local representations
of the Hamiltonian  in the classes of edge/corner degenerate differential operators. The existence of
such local coordinate systems
is not completely obvious. It turned out that such coordinates were known in the
physics literature for a
long time, cf.~the monographs \cite[pp.~1730ff]{MF53} and \cite[pp.~398ff]{MM65},
however not in the framework of the edge/corner degenerate calculus discussed below.

In a local neighbourhood $U_1$ of the $e-n$ singularity (i.e., $|\boldsymbol{x}_1|=0$ and therefore $|\boldsymbol{x}_2|=|\boldsymbol{x}_1-\boldsymbol{x}_2|=t)$, we define local hyperspherical coordinates
\begin{equation}
 \tilde{x}_1 = \sin\,r_1 \sin \theta_1 \cos \phi_1, \ \ \tilde{x}_2 = \sin\,r_1 \sin \theta_1
\sin \phi_1, \ \
 \tilde{x}_3 = \sin\,r_1 \cos \theta_1 ,
\label{polarX1}
\end{equation}
\begin{equation}
 \tilde{x}_4 = \cos\,r_1\sin \theta_2 \cos \phi_2, \ \ \tilde{x}_5 = \cos\,r_1\sin \theta_2
\sin \phi_2, \ \
 \tilde{x}_6 = \cos\,r_1 \cos \theta_2 ,
\label{polarY1}
\end{equation}
with respect to the projective coordinates $\tilde{x}_i := x_i/t$, $i=1,\ldots,6$,
on $S^5$.

One can consider (\ref{polarX1}) as polar coordinates on the stretched cone
$X^\wedge_1$,
with $r_1\in (0,\frac{\pi}{2}]$, $ \theta_1 \in(0,\pi)$, $\phi_1 \in[0, 2\pi)$. The remaining two angular
variables in (\ref{polarY1}),
with $ \theta_2 \in (0,\pi)$, $ \phi_2 \in [0, 2\pi)$, provide a spherical coordinate system on $Y_1$.
Similar coordinates can be constructed in a local neighbourhood of $U_2$.

For a neigbourhood of the $e-e$ cusp $U_3$, it is convenient to define center of
mass coordinates
\begin{gather*}
 z_1 = \frac{1}{\sqrt{2}} (x_1 - x_4), \ \ z_2 = \frac{1}{\sqrt{2}} (x_2 - x_5), \ \
 z_3 = \frac{1}{\sqrt{2}} (x_3 - x_6) \\
 z_4 = \frac{1}{\sqrt{2}} (x_1 + x_4), \ \ z_5 = \frac{1}{\sqrt{2}} (x_2 + x_5). \ \
 z_6 = \frac{1}{\sqrt{2}} (x_3 + x_6)
 \end{gather*}

We can now introduce with respect to $z_1,\ldots,z_6$ the same type of local polar
coordinates as
in the previous cases.

In the corner case, $t \rightarrow 0$, one can use the same coordinates, however, one has to consider $S^5$ globally, which means that the three local coordinate neighbourhoods form an atlas on the manifold.

\subsection{Metric tensor, Laplace-Beltrami operator and the Coulomb potential}
In this section we want to discuss the metric tensor, corresponding Laplace-Beltrami operator\footnote{ 
A basic reference for these differential geometric quantities is \cite{G63}.}
and the singular Coulomb potential in the edge and corner degenerate case.

The Riemannian metric on the five-sphere $S^5$ is induced by the Euclidean metric in
$\mathbb{R}^6$, i.e.,
\[
 ds^2 = \sum_{i=1}^6 dx_i^2 .
\]
Introducing the radial coordinate $t$, cf.~(\ref{t-def}), yields a metric of the form
\begin{equation}
 ds^2 = dt^2 + t^2 d\tilde{s}^2
\label{cornermetric}
\end{equation}
with
$
 d\tilde{s}^2 $
the induced metric on $S^5$.

For the previously defined polar coordinates on $S^5$ in a neigbourhood of
the edge $|\boldsymbol{x}_1|=0$
  the metric becomes
\begin{eqnarray}\label{edgemetric}
 d\tilde{s}^2  =  dr_1^2 + \sin^2 r_1 ( \underbrace{d\theta_1^2 + \sin^2 \theta_1 d
\phi_1^2 }_{=:g_{X_1}}) +
 \cos^2 r_1 ( \underbrace{d\theta_2^2 + \sin^2 \theta_2 d \phi_2^2}_{=:g_{Y_1}} ),
\end{eqnarray}
which has the form of a wedge metric\footnote{Note that a wedge metric has the general form
$
d r^2+r^2 g_X(r,y)+dy^2
$
where $r$ and $y$ denote the axial and edge variables, respectively, and $g_X$ is  a Riemannian metric on $X$, depending smoothly on $y$ and $r$ up to $r=0$.}.

Inserting (\ref{edgemetric}) in (\ref{cornermetric}) one gets the corresponding metric tensor in $\mathbb{R}^6$:
\begin{equation}\label{edgecornermetric}
 ds^2 = dt^2 + t^2 [  dr_1^2 + \sin^2 r_1 g_{X_1} +
 \cos^2 r_1 g_{Y_1}].
\end{equation}

 In the following we have to distinguish between the edge and corner degenerate case:

 The metric (\ref{edgecornermetric}) is already in the appropriate form for a corner degeneracy\footnote{
 Note that a corner metric in general is of the form $dt^2 +t^2(dr^2 +r^2 g_X(t,r,y)+dy^2)$ where $t$ and $r$ denote the corner and cone axial variables, respectively, $y$ the edge variable and $g_X$ is a Riemannian metric on $X$, depending smoothly on $y$ and $t,r$ up to $t=r=0$.}, however,
 in the edge degenerate case $t$ plays the role of an edge variable with values in an open interval $(a,b)\in\mathbb{R}_+$ for $a>0$.  The corresponding edge metric (\ref{edgecornermetric}) takes the form
 \[
 t^2  d r_1^2+t^2 \sin^2 r_1 g_{X_1}+\underbrace{dt^2+t^2 \cos^2 r_1 g_{Y_1}}_{\text{metric on the edge}}.
 \]

The Laplacian (\ref{Laplace}) in $\mathbb{R}^6$ can be expressed in hyperspherical coordinates according to
\begin{equation}\label{delta12}
\Delta_1+\Delta_2=t^{-2}[(-t\partial_t)^2-4(-t\partial_t)+\tilde{\Delta}]
\end{equation}
where $\tilde{\Delta}$ denotes the Laplace-Beltrami operator on $S^5$ which is given near the edge $|\boldsymbol{x}_1|=0$ by
\begin{equation}\label{tildedelta}
\tilde{\Delta}=r_1^{-2}[(-r_1\partial_{r_1})^2+h(r_1)(-r_1\partial_{r_1})+\frac{r_1^2}{\sin^2 r_1}\Delta_{X_1}+
\frac{r_1^2}{\cos^2 r_1}\Delta_{Y_1}]
\end{equation}
with  Laplace-Beltrami operators $\Delta_{X_1}$ and $\Delta_{Y_1}$ on unit
spheres\footnote{For example, denoting the spherical coordinates on $S^2 \approx X_1$ by $(\theta_1,\phi_1)$ we have
\[
\Delta_{X_1}=\partial^2 _{\theta_1}+\mbox{ctan} \, \theta_1\partial_{\theta_1}+\frac{1}{\sin^2 \theta_1} \partial^2_{\phi_1}.
\]}
 $S^2\approx X_1$ and $S^2\approx Y_1,$ respectively, and
\[
1+2r_1\mbox{tan} \, r_1-2r_1\mbox{ctan}\, r_1=:h(r_1)\in C^\infty \big([0,\pi/2)\big).
\]

Finally, we have to consider the Coulomb potential in hyperspherical coordinates.

In a neighbourhood of the edge $|\boldsymbol{x}_1|=0$, the Coulomb potential $V$ is given in the following form
\begin{equation}\label{potential}
V := -\frac{2}{|\boldsymbol{x}_1|}-\frac{2}{|\boldsymbol{x}_2|}+\frac{1}{|\boldsymbol{x}_1-\boldsymbol{x}_2|}
=t^{-1} r_1^{-1} v(r_1,\theta_1,\phi_1,\theta_2,\phi_2)
\end{equation}
with
\[
 v = -\frac{2 r_1}{\sin r_1} -\frac{2 r_1}{\cos r_1} +
 \frac{ r_{1}}{\sqrt{1-\sin(2r_{1})
 [ \cos \theta_1 \cos \theta_2 + \sin \theta_1 \sin \theta_2 \cos(\phi_1-\phi_2)]}}.
\]
For the singular calculus discussed below it is important to note that $v$ is smooth with respect to $r_1$ up to $r_1=0.$

In this section explicit formulas are given for the edge $|\boldsymbol{x}_1|=0$; analogous expressions can be given for the remaining $e-n$ and $e-e$ singularities.

\subsection{Fuchs type differential operators for edge and corner singularities}
The Fuchs type differential operator of order $m\in\mathbb{N}$ corresponding to the edge degenerate case has the general form
\begin{equation}\label{*}
A=r^{-m}\sum_{j+|\alpha|\leq m} a_{j\alpha}(r,y)(-r\partial_r)^j (rD_y)^\alpha.
\end{equation}
Here $r\in\mathbb{R}_+$ denotes the distance to the edge and $y$ is a $q$-dimensional variable, varying on an edge, say, $\Omega\subset \mathbb{R}^q$ open. The coefficients $a_{j\alpha}(r,y)$ take values in differential operators of order $m-(j+|\alpha|)$ on the base $X$ of the cone (which is in our case a sphere) and are smooth in the respective variables up to $r=0$.

The Fuchs type differential operator of order $m\in\mathbb{N}$ corresponding to the corner degenerate case has the general form
\begin{equation}\label{**}
A=t^{-m}\sum_{k=0}^m b_k(t) (-t\partial_t)^k.
\end{equation}
Here $t\in\mathbb{R}_+$ denotes the distance to the corner point and the coefficients $b_k(t)$ have values in operators of the form (\ref{*}) for any $t\in\mathbb{R}_+$. Inserting into (\ref{**}) the values of $b_k(t)$ it becomes degenerate with respect both $t$ and $r$ close to zero; for instance, we obtain
\begin{equation}\label{***}
A=t^{-m}r^{-m}\sum_{k+j+|\alpha|\leq m} a_{kj\alpha}(t,r,y)(-r\partial_r)^j (rD_y)^\alpha (-rt\partial_t)^k
\end{equation}
with coefficients $a_{kj\alpha}(t,r,y)$ smooth up to $t=r=0.$

\subsection{The Hamiltonian as an edge/corner degenerate operator}
In the previous section we have discussed some formal aspects of edge/corner degenerate differential operators. It is now straightforward to represent the Hamiltonian of the helium atom in edge and corner degenerate form using the results of Section 2.3.

Let us first consider the corner degenerate case which can be obtained by inserting (\ref{tildedelta}) in (\ref{delta12}) and adding the potential (\ref{potential}):
\begin{eqnarray}
H|_{U_1} & = & t^{-2}r_1^{-2}\Big[ -\frac{1}{2}(-r_1 t \partial_t)^2+2r_1(-r_1 t \partial_t)-\frac{1}{2}(-r_1\partial_{r_1})^2-\frac{h(r_1)}{2}(-r_1\partial_{r_1})\Big.\nonumber\\
& & -\frac{1}{2\cos^2 r_1}(r_1\partial_{\theta_2})^2-\frac{r_1 \mbox{ctan} \, \theta_2}{2\cos^2 r_1}(r_1\partial_{\theta_2})-\frac{1}{2\sin^2 \theta_2 \cos^2 r_1}(r_1\partial_{\phi_2})^2\nonumber\\
& & - \Big. \frac{r_1^2}{2\sin^2 r_1}\Delta_{X_1}+tr_1 v\Big] . \label{cornercase}
\end{eqnarray}
It is easy to see that (\ref{cornercase}) is already of the corner degenerate form (\ref{***}).

In the edge degenerate case $t$ becomes an edge variable. Accordingly one has to rearrange  the Hamiltonian
(\ref{cornercase}) in order to get an expression of the form (\ref{*}), i.e.,
\begin{eqnarray}
 H|_{U_1} & = & r^{-2}_1 \Big[ -\frac{1}{2 t^2} (- r_1
{\partial_{r_1}} )^2
 - \frac{h(r_1)}{2t^2} (- r_1 {\partial_{r_1}})
 -\frac{1}{2} (r_1 {\partial_t} )^2
 - \frac{5 r_1}{2t} (r_1 {\partial_t} )  \nonumber\\
 & & - \frac{1}{2t^2 \cos^2 r_1} (r_1 {\partial_{\theta_2}}
)^2
 - \frac{r_1 \mbox{ctan} \, \theta_2}{2t^2 \cos^2 r_1} (r_1 {\partial_
{\theta_2}} )
- \frac{1}{2t^2 \sin^2 \theta_2 \cos^2 r_1} (r_1 {\partial_
{\phi_2}} )^2 \nonumber\\
 & & - \frac{r_1^2}{2t^2 \sin^2 r_1} \Delta_{X_1} + \frac{r_1}{t} v \Big] .\label{He1n.edge}
\end{eqnarray}

\section{Ellipticity of the Hamiltonian in the edge degenerate case}
In the present work we restrict us to the edge degenerate case. Our main result presented in this section is the ellipticity of the Hamiltonian operator (\ref{He1n.edge}) as an operator in the edge algebra. This requires a symbolic hierarchy which will be briefly discussed in the following Section 3.1.
\subsection{Symbolic hierarchy and ellipticity in the edge degenerate sense}

The ellipticity of a differential operator on a manifold with edge singularity, cf.~(\ref{*}), (with a closed compact cone basis) is characterized by a pair of symbols
\[
\big(\sigma_\psi, \sigma_\wedge\big),
\]
where $\sigma_\psi$ is the usual homogeneous principal symbol of the  operator and $\sigma_\wedge$ denotes the so-called \emph{principal edge} symbol. For a detailed outline of the theory of ellipticity we refer to the monographs \cite{harutyunyanschulze} or \cite{Schulze98}.

The ellipticity conditions of an operator in the form (\ref{*}) with respect to $\sigma_\psi$ are
\begin{itemize}
\item  ellipticity in the usual sense outside of the singularity, i.e.,
\[
\sigma_\psi(A)(r,x,y,\rho,\xi,\eta)\not=0\quad\text{for all}\quad (r,x,y)\in X^\wedge\times \Omega\quad\text{and}\quad (\rho,\xi,\eta)\not=0;
\]
\item
\[
\tilde{\sigma}_\psi(A)(r,x,y,\rho,\xi,\eta)=r^m\sigma_\psi(A)(r,x,y,r^{-1}\rho,\xi,r^{-1}\eta)\not=0\quad\text{up to}\quad r=0.
\]
\end{itemize}

The second symbol $\sigma_\wedge(A),$ defined by
\[
\sigma_\wedge(A)(y,\eta):=r^{-m}\sum_{j+|\alpha|\leq m} a_{j\alpha}(0,y)(-r\partial_r)^j (r\eta)^\alpha,
\]
corresponds to a parameter-dependent operator family in the cone algebra. Ellipticity with respect to $\sigma_\wedge$ requires that $\sigma_\wedge(A)$ represents isomorphisms between weighted Sobolev spaces $\mathcal{K}^{s,\gamma}(X^\wedge)$
\footnote{In the particular case $X=S^n$ (which is the case in our application) the $\mathcal{K}^{s,\gamma}(X^\wedge)$ spaces are defined as
\[
\mathcal{K}^{s,\gamma}(X^\wedge):=\omega \mathcal{H}^{s,\gamma}(X^\wedge)+(1-\omega)H^s(\mathbb{R}^{1+n})
\]
for a cut-off function $\omega$ (i.e., $\omega\in C_0^\infty(\overline{\mathbb{R}}_+)$ such that $\omega(r)=1$ near $r=0$); here $\mathcal{H}^{s,\gamma}(X^\wedge)$ for $s\in\mathbb{N}_0$ is defined to be the set of all $u(r,x)\in r^{\gamma-\frac{n}{2}}L^2 (X^\wedge)$ such that $(r\partial_r)^j D_j u\in r^{\gamma-\frac{n}{2}}L^2(X^\wedge)$ for all $D_j\in \mathrm{Diff}^{s-j}(X), 0\leq j\leq s.$ The definition for $s\in\mathbb{R}$ follows by duality and complex interpolation.}:
\[
\sigma_\wedge(A)(y,\eta):\mathcal{K}^{s,\gamma}(X^\wedge)\rightarrow \mathcal{K}^{s-m,\gamma-m}(X^\wedge)
\]
for any $y\in\Omega, \eta\not=0.$
\subsection{Ellipticity of the Hamiltonian}
In our proof of the ellipticity of the Hamiltonian of the helium atom in the edge degenerate case, it is sufficient
to consider the $e-n$ singularity ${\bf x}_1=0$. The proofs in the remaining two cases are completely analogous.
In order to study the symbolic structure of the operator (\ref{He1n.edge}), we assign to variables
\[
r_1,\quad t,\quad (\theta_1,\phi_1),\quad (\theta_2,\phi_2)
\]
the covariables
\[
\rho,\quad \tau,\quad (\Theta_1,\Phi_1),\quad (\Theta_2,\Phi_2).
\]
The corresponding principal symbols are given by
\begin{eqnarray*}
\sigma_\psi(H|_{U_1}) & = & r_1^{-2}\Big[ \frac{(r_1 \rho)^2}{2t^2}+\frac{(r_1\tau)^2}{2}+\frac{(r_1 \Theta_2)^2}{2t^2 \cos^2 r_1}+\frac{(r_1\Phi_2)^2}{2t^2 \sin^2 \theta_2 \cos^2 r_1}\\ & & +\frac{r_1^2 \Theta_1^2}{2t^2 \sin^2 r_1}+\frac{r_1^2 \Phi_1^2}{2t^2 \sin^2 r_1 \sin^2 \theta_1}\Big];\\
\tilde{\sigma}_\psi (H|_{U_1}) & = & \frac{\rho^2}{2t^2}+\frac{\tau^2}{2}+\frac{\Theta_2^2}{2t^2 \cos^2 r_1}+\frac{\Phi_2^2}{2t^2 \sin^2 \theta_2 \cos^2 r_1}+\frac{r_1^2 \Theta_1^2}{2t^2 \sin^2 r_1}\\ & & +\frac{r_1^2 \Phi_1^2}{2t^2 \sin^2 r_1 \sin^2 \theta_1};\\
\sigma_\wedge(H|_{U_1}) & = & r_1^{-2}\Big[ -\frac{1}{2t^2}(-r_1\partial_{r_1})^2 +
\frac{1}{2t^2}(-r_1\partial_{r_1})+\frac{(r_1 \tau)^2}{2}+
\frac{(r_1 \Theta_2)^2}{2t^2}\\ & & +\frac{(r_1 \Phi_2)^2}{2t^2\sin^2 \theta_2}-\frac{\Delta_{X_1}}{2t^2}\Big],
\end{eqnarray*}
and ellipticity conditions with respect to the symbols $\sigma_\psi, \tilde{\sigma}_\psi$ are obviously satisfied.

In the following it is crucial to find values of $\gamma$ for which $\sigma_\wedge(H|_{U_1})$ is an isomorphism between the corresponding $\mathcal{K}^{s,\gamma}((S^2)^\wedge)$ spaces\footnote{Although we restrict the Hamiltonian to a bounded 
neighbourhood of the edge singularity, our notion of ellipticity requires a natural extension of $\sigma_\wedge$ 
to the infinite open stretched cone $(S^2)^\wedge$, cf.~\cite{Schulze98}.}
That such values actually exist is the result of Lemma \ref{lem1} below\footnote{We want to emphasize  that contrary to the general ellipticity theory developed in \cite{Schulze98} or \cite{harutyunyanschulze}, we do not want to consider possible (finite-dimensional) extensions of the Hamiltonian operator in order to achieve bijectivity for $\sigma_\wedge$.}. In the proof of this lemma we need the following result:

\begin{lemma}
The principal conormal symbol\footnote{The {\em principal conormal} symbol corresponding to a Fuchs type operator of the form
$A = r^{-m} \sum_{j=0}^m a_j(r)(-r\partial_r)^j$ is defined as $\sigma_c(A) (w) = \sum_{j=0}^m a_j(0) w^j$.} 
of $\sigma_\wedge(H|_{U_1})$ as an elliptic operator
$\sigma_c(\sigma_\wedge(H|_{U_1}))(w): H^s(S^2) \rightarrow H^{s-2}(S^2)$
induces an isomorphism between the Sobolev spaces  provided $w\notin
\mathbb{Z}$.
\label{lemma1}
\end{lemma}
A simple proof is given in \cite[Lemma 1]{FHSchSch}.

\begin{lemma}\label{lem1}
The principal edge symbol $\sigma_\wedge(H|_{U_1})$ defines Fredholm operators
\[
\sigma_\wedge(H|_{U_1})(t,(\theta_2,\phi_2),\tau, (\Theta_2,\Phi_2)):\mathcal{K}^{s,\gamma}((S^2)^\wedge)\rightarrow \mathcal{K}^{s-2,\gamma-2}((S^2)^\wedge)
\]
for $\gamma\notin \mathbb{Z}+\frac{1}{2},$ for all $(t,(\theta_2,\phi_2))\in\Omega\times S^2$ and all $(\tau, (\Theta_2,\Phi_2))\not=0$.

Furthermore, it represents isomorphisms for any $\gamma$ with $\frac{1}{2}<\gamma<\frac{3}{2}.$
\end{lemma}

\begin{proof}
The proof is given in two steps: first we show the Fredholm property of $\sigma_\wedge$ for any $\gamma\notin\mathbb{Z}+\frac{1}{2}$, and then we prove  the isomorphy for $\frac{1}{2}<\gamma<\frac{3}{2}$, which is a central point for our application.

The Fredholm property for $\gamma \notin \mathbb{Z} + \frac{1}{2}$ is an immediate
consequence of Lemma \ref{lemma1} and \cite[Section 8.2.5 Theorem 8]{ES97}. It only remains to check the
exit behaviour of $\sigma_\wedge(H|_{U_1})$
required by this theorem. To this end $\sigma_\wedge(H|_{U_1})$\footnote{In the following it is convenient  
to write the principal edge symbol as
\[
 \sigma_\wedge(H|_{U_1}) = - \frac{1}{2t^2} \left[ \partial^2_{r_1} + \frac{2}{r_1} \partial_{r_1}
 + \frac{\Delta_{X_1}}{r_1^2} + C(t,\tau,\Theta_2,\Phi_2) \right] ,
\]
where $C(t,\tau,\Theta_2,\Phi_2) := - (t \tau)^2 - \Theta_2^2 - \frac{\Phi_2^2}{\sin^2 \theta_2}$.}
can be pushed forward to a differential operator
$\bar{\sigma}_\wedge(H|_{U_1}) = - \frac{1}{2t^2} \left( \Delta +C \right)$ in
$\mathbb{R}^3$.
The corresponding exit symbols
\[
 \sigma_e(\bar{\sigma}_\wedge(H|_{U_1})) = \frac{1}{2t^2} \left( \xi_1^2 + \xi_2^2 +
\xi_3^2 -C \right) \ \ \mbox{and} \
 \sigma_{\psi,e}(\bar{\sigma}_\wedge(H|_{U_1})) = \frac{1}{2t^2} \left( \xi_1^2 + \xi_2^2 +
\xi_3^2 \right) ,
\]
have constant coefficients and satisfy for $(\tau,\Theta_2,\Phi_2) \neq 0$ the
required exit conditions, i.e.,
\[
 \sigma_e(\bar{\sigma}_\wedge(H|_{U_1})) \neq 0 \ \ \mbox{for all} \ \xi \in \mathbb{R}^3 \
\ \mbox{and} \
 \sigma_{\psi,e}(\bar{\sigma}_\wedge(H|_{U_1})) \neq 0 \ \ \mbox{for all} \ \xi \in
\mathbb{R}^3 \setminus \{0\} .
\]

In order to establish the isomorphism, we have to calculate the dimension of the
kernel and cokernel.
A nontrivial element $u \in \mathcal{K}^{s,\gamma}(X^\wedge)$ of the kernel of
$\sigma_\wedge(H|_{U_1})$ satisfies
the equation
\[
 \left[ \partial_{r_1}^2 + \frac{2}{r_1} \partial_{r_1}
 + \frac{\Delta_{X_1}}{r_1^2} + C(t,\tau,\Theta_2,\Phi_2) \right] u =0 .
\]
According to \cite[Section 8.2.5 Theorem 8]{ES97}, $u$ belongs to $\mathcal
{K}^{\infty,\gamma}((S^2)^\wedge)$
and therefore we have $u \in C^\infty((S^2)^\wedge)$. The remaining part
of the proof is
based on the separation of variables technique, cf.~\cite{Triebel}, where
for fixed $r_1 > 0$ we perform an expansion of the solution in terms of spherical
harmonics
\[
 u(r_1,\theta_1,\phi_1) = \sum_{l=0}^\infty \sum_{m=-l}^l \bigl( u(r_1,\cdot), Y_{lm}
\bigr)_{L_2(S^2)} Y_{lm}(\theta_1,\phi_1) ,
\]
which converges absolutely and uniformly on $S^2$. The spherical harmonics
are eigenfunctions of $\Delta_{X_1}$ with eigenvalues $-l(l+1)$ and represent a complete
orthogonal basis in $L_2(S^2)$.
Therefore
\[
 \sum_{l=0}^\infty \sum_{m=-l}^l \left[ \partial_{r_1}^2 +
\frac{2}{r_1} \partial_{r_1}
 - \frac{l(l+1)}{r_1^2} + C \right] \bigl( u(r_1,\cdot), Y_{lm} \bigr)_{L_2(S^2)}
Y_{lm}(\theta_1,\phi_1) =0 ,
\]
and $u$ is a nontrivial solution if and only if at least one of the ordinary
differential equations
\[
 \left[ d^2 + \frac{2}{r_1} d
 - \frac{l(l+1)}{r_1^2} + C \right] f_l(r_1) =0 , \ \ \mbox{with} \ l \in
\mathbb{N}_0 ,
\]
has a classical solution $f_l$. After the substitution $f_l(r_1) =
r_1^{-\frac{1}{2}} \chi(r_1)$, the equation
becomes
\[
 r_1^2 \chi''(r_1) + r_1 \chi'(r_1) + \left( C r_1^2 - (l+\tfrac{1}{2})^2 \right)
\chi(r_1) =0 .
\]
According to our assumption on the covariables of the edge, $(\tau,\Theta_2,\Phi_2)
\neq 0$, we have
$-C > 0$ and a change of variables $z := \sqrt{-C} r_1$ with $w(z) :=
\chi(z/\sqrt{-C})$ yields the final equation
\[
 z^2 w''(z) + z w'(z) - \left( z^2 + (l+\tfrac{1}{2})^2 \right) w(z) =0 .
\]
The classical solutions of this equation are the modified Bessel functions
$I_{\pm(l+\tfrac{1}{2})}$ and
$K_{l+\tfrac{1}{2}}$, cf.~\cite[p. 374]{AS}. These functions have asymptotic
expansions for $z \rightarrow \infty$
\[
 I_{\pm(l+\tfrac{1}{2})}(z) \sim \frac{e^z}{\sqrt{2 \pi z}} \left( 1 + \mathcal{
O}(\tfrac{1}{z}) \right)
 \ \ \mbox{and} \
 K_{l+\tfrac{1}{2}}(z) \sim e^{-z} \frac{\pi}{\sqrt{2 z}} \left( 1 + \mathcal{
O}(\tfrac{1}{z}) \right) ,
\]
and therefore $r_1^{-\frac{1}{2}}I_{\pm(l+\tfrac{1}{2})}(\sqrt{-C} r_1) Y_{lm}
\notin \mathcal{K}^{2,\gamma}((S^2)^\wedge)$
for all $\gamma \in \mathbb{R}$.
The functions $K_{l+\tfrac{1}{2}}$ have asymptotic expansions
\[
 K_{l+\tfrac{1}{2}}(z) \sim \tfrac{1}{2} \Gamma(l+\tfrac{1}{2}) \left( \tfrac{1}{2}
z \right)^{-(l+\tfrac{1}{2})}
 \ \ \mbox{for} \ z \rightarrow 0 .
\]
Together with the recursion relations for the derivatives, cf.~\cite{AS},
\[
 K'_{l+\tfrac{1}{2}}(z) = -\frac{1}{2} \left( K_{l-\tfrac{1}{2}}(z) +
K_{l+\tfrac{3}{2}}(z) \right)
\]
and
\[
 K''_{l+\tfrac{1}{2}}(z) = \frac{1}{4} \left( K_{l-\tfrac{3}{2}}(z) + 2
K_{l+\tfrac{1}{2}}(z) + K_{l+\tfrac{5}{2}}(z) \right) ,
\]
it is easy to see that
\[
 r_1^{-\frac{1}{2}} K_{l+\tfrac{1}{2}}(\sqrt{-C} r_1) Y_{lm} \notin \mathcal{
K}^{2,\gamma}((S^2)^\wedge) \ \ \mbox{for} \
\gamma \geq -l + \tfrac{1}{2} ,
\]
i.e., the kernel contains in particular no nontrivial solution for $\gamma \geq
\tfrac{1}{2}$ and
\begin{equation}
 \ker \sigma_\wedge(H|_{U_1}) = \left\{ K_{l+\tfrac{1}{2}} : l \in \mathbb{N}_0 \
\mbox{and} \
  0 \leq l < -\gamma + \frac{1}{2} \right\} .
\label{kerH}
\end{equation}

It remains to calculate the dimension of the cokernel of the principal edge symbol
$\sigma_\wedge(H|_{U_1})$
which equals the dimension of the kernel of the adjoint operator
$\sigma_\wedge(H|_{U_1})^\ast: \mathcal{K}^{0,2-\gamma}((S^2)^\wedge) \rightarrow \mathcal{
K}^{-2,-\gamma}((S^2)^\wedge)$.
Weighted Sobolev spaces with $s<0$ are defined via a non-degenerate scalar product,
cf.~\cite{Schulze98},
and the identification of $\mathcal{K}^{-s,-\gamma}((S^2)^\wedge)$ with the dual space
$\mathcal{K}^{s,\gamma}((S^2)^\wedge)^\ast$
follows from similar reasoning as for standard Sobolev spaces, cf.~\cite{Yosida}.
By the same argument as before every solution of $\sigma_\wedge(H|_{U_1})^\ast u=0$ belongs
to
$\mathcal{K}^{\infty,2-\gamma}((S^2)^\wedge)$. The adjoint differential operator
$\sigma_\wedge(H|_{U_1})^\ast$ therefore coincides with $\sigma_\wedge(H|_{U_1})$
and it follows from our previous computation that the kernel of
$\sigma_\wedge(H|_{U_1})^\ast$
contains no nontrivial solution for $\gamma \leq \tfrac{3}{2}$. In general we have
\begin{equation}
 \ker \sigma_\wedge(H|_{U_1})^\ast = \left\{ K_{l+\tfrac{1}{2}} : l \in \mathbb{N}_0 \
\mbox{and} \
  0 \leq l < \gamma - \frac{3}{2} \right\} .
\label{kerH*}
\end{equation}
Comparing (\ref{kerH}) and (\ref{kerH*}), we get the isomorphy of $\sigma_\wedge(H|_{U_1})$ 
for $\frac{1}{2} < \gamma < \frac{3}{2}$.
\end{proof}

We want to close with a remark concerning
the appropriate choice of $\gamma$ which is a crucial issue of our approach to construct
the asymptotic parametrix and Green operator explicitly. Restrictions on $\gamma$
are imposed by ellipticity of the edge degenerate operator according to Lemma \ref{lem1}. 
It is however not obvious that
these values of $\gamma$ are appropriate for eigenfunctions of the Hamiltonian.
The spin-free eigenfunctions of the helium atom are either symmetric or antisymmetric
with respect to an interchange of the electron coordinates. It is a standard
result in the spectral theory of Schr\"odinger operators that the lowest eigenstate $\Psi_0$
must be nondegenerate and strictly positive, cf.~\cite[Theorem XII.46]{RSIV} and following Corollary.
In particular \cite[Theorem 1.2]{HO2S94} gives
\[
 \Psi_0 = c \left( 1 + O(s \ln s) \right) , \ \ \mbox{with} \ c>0 ,
\]
where $s$ denotes the Euclidean distance between two electrons in $\mathbb{R}^6$. The distance parameter
$r$ in hyperspherical coordinates refers to the geodesic distance to an edge on $S^5$.
A simple calculations shows that $s = tr + O(tr^3)$ for $r \rightarrow 0$. Therefore,
one gets $\Psi_0(t,r,\phi_1,\theta_1,\phi_2,\theta_2)|_{t,\phi_2,\theta_2= \
\mbox{\scriptsize const.}} \notin \mathcal{K}^{s,\gamma} \bigl( (S^2)^\wedge \bigr)$ for $\gamma \geq \frac{3}{2}$.
Vice versa, a natural lower bound for $\gamma$ is derived from the following consideration.
The Hamiltonian of the helium atom can be defined as a selfadjoint operator corresponding
to the quadratic form
\[
 h(u) := \tfrac{1}{2} \langle \nabla u, \nabla u \rangle + \langle V u, u \rangle
\]
with form domain $Q(h) = Q(\Delta) \cap Q(V)$. As an edge degenerate operator, the Hamiltonian acts on
an egde Sobolev space 
$\mathcal{W}_{P,\mbox{\footnotesize comp}}^{s} \bigl( Y, \omega \mathcal{K}^{s,\gamma}_P \bigl( (S^2)^\wedge \bigr) \bigr)$
for some asymptotic type $P$, cf.~\cite{Schulze98}. 
In order to check whether
$\mathcal{W}_{P,\mbox{\footnotesize comp}}^{s} \bigl( Y, \omega \mathcal{K}^{s,\gamma}_P \bigl( (S^2)^\wedge \bigr) \bigr)
\subset Q(V)$ it is sufficient to consider $\alpha := \mbox{max} \{\Re p_j \}$ with $\alpha < \frac{3}{2} -\gamma$.
In a neighbourhood of a Coulomb singularity one gets $\alpha < 1$ and therefore we can restrict ourselves
to $\mathcal{W}_{P,\mbox{\footnotesize comp}}^{s} \bigl( Y, \omega \mathcal{K}^{s,\gamma}_P \bigl( (S^2)^\wedge \bigr) \bigr)$
with $\gamma > \frac{1}{2}$.

\end{document}